\documentclass{amsart}
\usepackage{amsmath}
\usepackage{amssymb}
\usepackage{amsthm}
\usepackage{enumerate}
\usepackage[dvips]{graphicx}
\theoremstyle{definition}
\newtheorem{definition}{Definition}[section]
\theoremstyle{plain}
\newtheorem{lemma}[definition]{Lemma}
\newtheorem{theorem}[definition]{Theorem}

\newtheorem{corollary}[definition]{Corollary}
\theoremstyle{remark}

\makeatletter
\@namedef{subjclassname@2010}{
  \textup{2010} Mathematics Subject Classification}
\makeatother

\begin{document}

\title[uniform local definable cell decomposition]{Uniform local definable cell decomposition for locally o-minimal expansion of the group of reals} 
\author[M. Fujita]{Masato Fujita}
\address{Department of Liberal Arts,
Japan Coast Guard Academy,
5-1 Wakaba-cho, Kure, Hiroshima 737-8512, Japan}
\email{fujita.masato.p34@kyoto-u.jp}

\begin{abstract}
We demonstrate the following uniform local definable cell decomposition theorem in this paper.

Consider a structure $\mathcal M = (M, <,0,+, \ldots)$ elementarily equivalent to a locally o-minimal expansion of the group of reals $(\mathbb R, <,0,+)$. 
Let $\{A_\lambda\}_{\lambda\in\Lambda}$ be a finite family of definable subsets of $M^{m+n}$.
There exist an open box $B$ in $M^n$ containing the origin and a finite partition of definable sets $M^m \times B = X_1 \cup \ldots \cup X_k$ such that $B=(X_1)_b \cup \ldots \cup (X_k)_b$ is a definable cell decomposition of $B$ for any $b \in M^m$ and $X_i \cap A_\lambda = \emptyset$ or $X_i \subset A_\lambda$ for any $1 \leq i \leq k$ and $\lambda \in \Lambda$.
Here, the notation $S_b$ denotes the fiber of a definable subset $S$ of $M^{m+n}$ at $b \in M^m$.
\end{abstract}



\maketitle

\section{Introduction}\label{sec:intro}
The author introduced the notion of uniform local o-minimality of the second kind and investigated its basic properties in \cite{Fuji}.
A definably complete uniformly locally o-minimal structure of the second kind admits the following local definable cell decomposition:

\begin{quotation}
Let $\{A_\lambda\}_{\lambda\in\Lambda}$ be a finite family of definable subsets of $M^n$.
For any point $a \in M^n$, there exist an open box $B$ containing the point $a$ and a definable cell decomposition of $B$ partitioning the finite family $\{B \cap A_\lambda\;|\; \lambda \in \Lambda \text{ and }  B \cap A_\lambda \not= \emptyset\}$.
\end{quotation}

The purpose of this paper is to extend the above theorem when the structure is a structure $\mathcal M = (M, <,0,+, \ldots)$ elementarily equivalent to a locally o-minimal expansion of the group of reals $(\mathbb R, <,0,+)$. 
Our main theorem is as follows:

\begin{theorem}[Uniform local definable cell decomposition]\label{thm:main}
Consider a structure $\mathcal M = (M, <,0,+, \ldots)$ elementarily equivalent to a locally o-minimal expansion of the group of reals $(\mathbb R, <,0,+)$.
Let $\{A_\lambda\}_{\lambda\in\Lambda}$ be a finite family of definable subsets of $M^{m+n}$.
There exist an open box $B$ in $M^n$ containing the origin and a finite partition of definable sets 
\begin{equation*}
M^m \times B = X_1 \cup \ldots \cup X_k
\end{equation*}
such that $B=(X_1)_b \cup \ldots \cup (X_k)_b$ is a definable cell decomposition of $B$ for any $b \in M^m$ and $X_i \cap A_\lambda = \emptyset$ or $X_i \subset A_\lambda$ for any $1 \leq i \leq k$ and $\lambda \in \Lambda$.
Here, the notation $S_b$ denotes the fiber of a definable subset $S$ of $M^{m+n}$ at $b \in M^m$.
\end{theorem}

In an locally o-minimal structure which admits a local definable cell decomposition, the image of a definable set under a definable map may be of dimension greater than the dimension of the original definable set as in illustrated in \cite[Remark 5.5]{Fuji}.
However, the image is of dimension not greater then the original set when the universe is the set of reals.
It is demonstrated in Section \ref{sec:dim}.
In Section \ref{sec:multi}, we define multi-cells and show that any definable set is partitioned into finite multi-cells when the structure is a locally o-minimal expansion of the group of reals.
We finally demonstrate the main theorem in Section \ref{sec:udcd}.
We first demonstrate the main theorem when the structure is a locally o-minimal expansion of the group of reals.
Using a standard model-theoretic argument, we extend the theorem to the case in which the structure is a structure elementarily equivalent to a locally o-minimal expansion of the group of reals.

We introduce the terms and notations used in this paper.
The term `definable' means `definable in the given structure with parameters' in this paper.
For a linearly ordered structure $\mathcal M=(M,<,\ldots)$, an open interval is a definable set of the form $\{x \in \mathbb R\;|\; a < x < b\}$ for some $a,b \in M$.
It is denoted by $]a,b[$ in this paper.
An open box in $M^n$ is the direct product of $n$ open intervals.
Let $A$ be a subset of a topological space.
The notations $\operatorname{int}(A)$ and $\overline{A}$ denote the interior and the closure of the set $A$, respectively.
The boundary $\operatorname{bd}(A)$ of $A$ is given by $\overline{A} \setminus  \operatorname{int}(A)$.
The frontier $\partial A$ of $A$ is defined by $\overline{A} \setminus A$.
The notation $|S|$ denotes the cardinality of a set $S$.
It also denotes the absolute value of an element.
The abuse of notation will not confuse readers.

\section{Dimension of the image}\label{sec:dim}
Consider a locally o-minimal structure whose universe is the set of reals $\mathbb R$.
Note that the structure in consideration is strongly locally o-minimal by \cite[Corollary 3.4]{TV}.
It admits local definable cell decomposition by \cite[Theorem 4.2]{Fuji} or \cite[Proposition 13]{KTTT}.
Therefore, the dimension of a definable set is well-defined by \cite[Section 5]{Fuji}.

In general, the image of a definable set under a definable map may be of dimension greater than the original definable set as illustrated in \cite[Example 5.2]{Fuji}.
We show that the image is of dimension not greater than the original definable set when the structure is a locally o-minimal structure whose universe is the set of reals $\mathbb R$.

We first demonstrate three lemmas necessary for the proof of the main result in this section.

\begin{lemma}\label{lem:dim_pre1}
Consider a locally o-minimal structure whose universe is the set of reals $\mathbb R$.
Let $I$ be an open interval in $\mathbb R$. 
Let $\{X_j\}_{j \in J}$ be a family of definable subsets of $I$ such that the index set $J$ is countable and all $X_j$ do not contain open intervals.
Then, the  union of all $X_j$ does not coincides with $I$.
\end{lemma}
\begin{proof}
For any $j \in J$ and $x \in \mathbb Q$, there exists an open subinterval $U_{j,x}$ of $I$ such that $U_{j,x}$ is an open neighborhood of $x$ and $X_j \cap U_{j,x}$ consists of finite points because the structure is locally o-minimal.
We have
\begin{align*}
\left| \bigcup_{j \in J} X_j \right| \leq |J| \cdot |\mathbb Q| \cdot \aleph_0 = \max \{ |J|, |\mathbb Q| , \aleph_0 \} < |I|\text{.}
\end{align*}
Hence, we have $\bigcup_{j \in J} X_j \not=I$.
\end{proof}

\begin{lemma}\label{lem:dim_pre2}
Consider a locally o-minimal structure whose universe is the set of reals $\mathbb R$.
Let $X$ be a definable subset of $\mathbb R^n$ of dimension $d$. 
Let $\{X_j\}_{j \in J}$ be a family of definable subsets of $X$ of dimension smaller than $d$ such that the index set $J$ is countable.
Then, we have $\bigcup_{j \in J} X_j \not=X$.
\end{lemma}
\begin{proof}
We demonstrate the lemma by the induction on $d$.
The lemma is obvious when $d=0$.

We next consider the case in which $d>0$.
There exist an open box $B \subset \mathbb R^d$ and a definable map $f:B \rightarrow X$ which is definably homeomorphic onto its image by the definition of dimension.
We may assume that $X$ is an open box $B$ by considering the preimages of $X$ and $X_j$ under $f$ by \cite[Corollary 5.2, Theorem 5.4]{Fuji}.
We may also assume that $B$ is bounded without loss of generality.

Let $\pi:\mathbb R^d \rightarrow \mathbb R^{d-1}$ be the projection forgetting the last coordinate.
Set $B_1 = \pi(B)$ and we have $B=B_1 \times I$ for some open interval $I$.
We consider the sets
\begin{align*}
Y_j = \{x \in B_1\;|\; \pi^{-1}(x) \cap X_j \text{ contains an open interval}\}
\end{align*}
for all $j \in J$.
They are definable and of dimension smaller than $d-1$. 
In fact, for any $x \in \overline{B}$, there exist an open box $B_x$ in $\mathbb R^{d-1}$ and an open interval $I_x$ such that $x \in B_x \times I_x$ and the definable set $Y_{jx} = \{x' \in B_x\;|\; \pi^{-1}(x') \cap X_j \cap I_x\text{ contains an open interval}\}$ is definable and of dimension smaller than $d-1$ by \cite[Lemma 5.4]{Fuji}.
Since $\overline{B}$ is compact, there exist finite points $x_1, \ldots, x_N$ in $\overline{B}$ such that $\{B_{x_i} \times I_{x_i}\}_{i=1}^N$ is an open covering of $\overline{B}$.
We have $Y_j = \bigcup_{i=1}^N Y_{jx_i}$ for any $j \in J$.
Therefore, $Y_j$ is definable and of dimension smaller than $d-1$ by \cite[Corollary 5.4(ii)]{Fuji}.

There exists a point $x \in B_1 \setminus \left( \bigcup_{j \in J} Y_j \right)$ by the induction hypothesis.
Since $x \not\in \bigcup_{j \in J} Y_j$, the intersection of $\pi^{-1}(x)$ with $X_j$ does not contain an open interval for any $j \in J$.
Therefore, we can take $y \in (\{x\} \times I) \setminus \bigcup_{j \in J} X_j$ by Lemma \ref{lem:dim_pre1}.
The point $(x,y) \in X=B$ is not contained in $\bigcup_{j \in J} X_j$.
\end{proof}

\begin{lemma}\label{lem:dim_pre3}
Consider a locally o-minimal structure whose universe is the set of reals $\mathbb R$.
Let $X$ be a definable subset of $\mathbb R^{m+n}$ and $\pi: \mathbb R^{m+n} \rightarrow \mathbb R^m$ be a coordinate projection.
Assume that  the fibers $X_x = \pi^{-1}(x) \cap X$ are of dimension $\leq 0$ for all $x \in \mathbb R^m$.
Then, we have $\dim X \leq \dim \pi(X)$.
\end{lemma}
\begin{proof}
For any $(a,b) \in \mathbb R^m \times \mathbb R^n$, there exist open boxes $B_a \subset \mathbb R^m$ and $B_b \subset  \mathbb R^n$ with $(a,b) \in B_a \times B_b$ and $\dim (X \cap (B_a \times B_b)) = \dim \pi(X \cap (B_a \times B_b))$ by \cite[Lemma 5.4]{Fuji}.
We have $\dim \pi(X \cap (B_a \times B_b)) \leq \dim \pi(X)$ by \cite[Lemma 5.1]{Fuji}.
On the other hand, we have $\dim (X) = \displaystyle\sup_{(a,b) \in \mathbb R^m \times \mathbb R^n} \dim (X \cap (B_a \times B_b))$ by \cite[Corollary 5.3]{Fuji}.
We have finished the proof.
\end{proof}

The following theorem is the main theorem of this section.
\begin{theorem}\label{thm:dim}
Consider a locally o-minimal structure whose universe is the set of reals $\mathbb R$.
Let $X$ be a definable set and $f:X \rightarrow \mathbb R^n$ be a definable map.
Then, we have 
\begin{equation*}
\dim(f(X)) \leq \dim(X)\text{.}
\end{equation*}
\end{theorem}
\begin{proof}
Let $X$ be a definable subset of $\mathbb R^m$.
We demonstrate the theorem by the induction on $\dim(X)$.

We first prove the theorem when $\dim(X)=0$.
We lead to a contradiction assuming that $\dim(f(X)) \geq 1$.
There exists an open interval $I$ and a definable map $g:I \rightarrow f(X)$ which is definably homeomorphic onto its image.
In particular, we have
\begin{equation*}
|I| = |g(I)| \leq |f(X)| \text{.}
\end{equation*}
On the other hand, the set $X$ is a discrete definable set because $\dim(X)=0$.
There exists a definable open neighborhood $U_x$ of $x$ such that $X \cap U_x$ is a finite set for any $x \in \mathbb Q^m$.
We get 
\begin{equation*}
|X| \leq  |\mathbb Q^m| \cdot \aleph_0  =\aleph_0 \text{.}
\end{equation*}
We finally obtain $|I| \leq |f(X)| \leq |X| \leq \aleph_0$.
Contradiction.

We next consider the case in which $\dim(X)>0$.
Set $d=\dim(X)$.
We lead to a contradiction assuming that $\dim(f(X)) \geq d+1$.
We can reduce to the case in which the image $f(X)$ is an open box $B$ of dimension $d+1$.
In fact, there exists a definable map $g:B \rightarrow f(X)$ for some open box $B$ in $\mathbb R^{d+1}$ and the map $g$ is a definable homeomorphism onto its image.
Set $Y=f^{-1}(g(B))$ and $h=g^{-1} \circ f|_Y:Y \rightarrow B$.
The map $h$ is clearly onto.
If $\dim(Y) <d$, then we have $d+1 = \dim B = \dim(h(Y)) \leq \dim(Y)<d$ by the induction hypothesis.
Contradiction.
We therefore get $\dim(Y)=d$.
We may assume that $f(X)=B$ by considering $Y$ and $h$ in place of $X$ and $f$.

We next reduce to the case in which the map $f$ is the restriction of a coordinate projection.
Consider the graph $G \subset \mathbb R^{m+d+1}$ of the definable map $f$.
Let $\pi:\mathbb R^{m+d+1} \rightarrow \mathbb R^{d+1}$ be the projection onto the last $d+1$ coordinates.
We have $\dim(G) \leq \dim(X)=d$ by Lemma \ref{lem:dim_pre3}.
The dimension of $G$ cannot be smaller than $d$ by the induction hypothesis in the same way as above because the restriction of $\pi$ to $G$ is a surjective map onto the open box $B$ of dimension $d+1$.
We get $\dim(G)=d$.
We may assume that $f:X \rightarrow B$ is the restriction of the projection $\pi:\mathbb R^{m+d+1} \rightarrow \mathbb R^{d+1}$ to $X$.

For any $x \in \mathbb Q^{m+d+1}$, there exists an open box $U_x$ such that $X \cap U_x$ is a finite union of cells because the structure admits local definable cell decomposition by \cite[Theorem 4.2]{Fuji}.
We have $\dim f(X \cap U_x) \leq \dim(X \cap U_x) \leq \dim X=d$ by \cite[Lemma 5.1, Corollary 5.4(ii), (iii)]{Fuji}.
Since $\{U_x\}_{x \in \mathbb Q^{m+d+1}}$ is an open cover of $\mathbb R^{m+d+1}$, we have $B=f(X)=\bigcup_{x  \in \mathbb Q^{m+d+1}} f(X \cap U_{x})$.
On the other hand, We obtain $B \not= \bigcup_{x  \in \mathbb Q^{m+d+1}} f(X \cap U_{x})$ by Lemma \ref{lem:dim_pre2}.
Contradiction.
\end{proof}

\begin{corollary}\label{cor:dim1}
Consider a locally o-minimal structure whose universe is the set of reals $\mathbb R$.
Let $X$ be a definable subset of $\mathbb R^m$ and $f:X \rightarrow \mathbb R^n$ be the restriction of a coordinate projection to $X$.
Assume further that $\dim(f^{-1}(x)) \leq 0$ for all $x \in \mathbb R^n$.
Then, we have $\dim(f(X)) = \dim(X)$.
\end{corollary}
\begin{proof}
Immediate from Lemma \ref{lem:dim_pre3} and Theorem \ref{thm:dim}.
\end{proof}

\begin{corollary}\label{cor:dim2}
Consider a locally o-minimal structure whose universe is the set of reals $\mathbb R$.
Let $X$ be a definable subset of $\mathbb R^m$ and $f:X \rightarrow \mathbb R$ be a definable function.
The notation $\mathcal D \subset X$ denotes the set of the points at which $f$ is discontinuous.
Then, we have $\dim(\mathcal D) < \dim(X)$.
\end{corollary}
\begin{proof}
Let $G$ be the graph of $f$.
We have $\dim(G)=\dim(X)$ by Corollary \ref{cor:dim1}.
Set $\mathcal E=\{(x,y) \in X \times \mathbb R\;|\; y=f(x) \text{ and } f \text{ is discontinuous at }x\}$.
We get $\dim(\mathcal E) < \dim(G)$ by \cite[Theorem 4.2, Corollary 5.3]{Fuji}.
Let $\pi:\mathbb R^{m+1} \rightarrow \mathbb R^m$ be the projection forgetting the last coordinate.
We have $\mathcal D = \pi(\mathcal E)$ by the definitions of $\mathcal D$ and $\mathcal E$.
We finally obtain $\dim(\mathcal D) = \dim(\pi(\mathcal E))  \leq \dim(\mathcal E) < \dim(G)=\dim(X)$ by Theorem \ref{thm:dim}.
\end{proof}

\section{Partition into multi-cells}\label{sec:multi}
A set definable in an o-minimal structure is a finite union of cells. 
See \cite{KPS,PS, vdD}.
A set definable in a definably complete uniformly locally o-minimal structure of the second kind is locally a finite union of cells by \cite[Theorem 4.2]{Fuji}, but it is not always true globally.
In this section, we define multi-cells and demonstrate that a definable set is decomposed into finite multi-cells if the structure is a locally o-minimal expansion of the group of reals $(\mathbb R, <,0,+)$. 

Fornasiero also defined multi-cells and demonstrated that a set definable in a definably complete locally o-minimal field is decomposed into finite multi-cells in \cite{F}.
We use the same term `multi-cell' in this paper.
Our definition of multi-cells is similar to but not the same as Fornasiero's.

We first define locally definable sets and investigate their basic properties.
\begin{definition}
Consider a locally o-minimal structure whose universe is the set of reals $\mathbb R$.
A \textit{locally definable subset} $X$ of $\mathbb R^n$ is a subset such that, for any  $x \in \mathbb R^n$, there exists an open box $U_x$ in $\mathbb R^n$ containing the point $x$ such that $X \cap U_x$ is definable.
A subset $X$ of $\mathbb R^{n+1}$ is called \textit{bounded in the last coordinate} if there exists a bounded open interval $I$ such that $X \subset \mathbb R^n \times I$.
\end{definition}

\begin{lemma}\label{lem:ld1}
Consider a locally o-minimal structure whose universe is the set of reals $\mathbb R$.
\begin{enumerate}
\item[(a)] A bounded locally definable set is definable.
\item[(b)] The closure of a locally definable set is locally definable.
\item[(c)] Any connected component of a locally definable set is locally definable.
\end{enumerate}
\end{lemma}
\begin{proof}
The assertions (a) and (b) are easy to prove.
We prove the assertion (c).
Take an arbitrary connect component $C$ of a locally definable subset $X$ of $\mathbb R^n$.
For any point $x \in \mathbb R^n$, there exists an open box $B$ containing $x$ such that $B \cap X$ is definable.
Shrinking $B$ if necessary, we may assume that $B \cap X$ is a finite union of cells by \cite[Theorem 4.2]{Fuji}.
The set $B \cap C$ is also definable because it is a finite union of cells.
It means that $C$ is locally definable.
\end{proof}

\begin{lemma}\label{lem:ld2}
Consider a locally o-minimal structure whose universe is the set of reals $\mathbb R$.
Let $X$ be a locally definable subset of $\mathbb R^{n+1}$ which is bounded in the last coordinate.
The image $\pi(X)$ is also locally definable, where $\pi:\mathbb R^{n+1} \rightarrow \mathbb R^n$ is the projection forgetting the last coordinate.
\end{lemma}
\begin{proof}
Take an arbitrary point $x \in \mathbb R^n$.
Since $X$ is bounded in the last coordinate, there exists an open interval $I$ with $X \subset \mathbb R^n \times I$.
Since $\{x\} \times \overline{I}$ is compact, there exists a finite cover by open boxes $\{B_i\}_{i=1}^m$ of $\{x\} \times \overline{I}$ such that $B_i \cap X$ is definable for any $1 \leq i \leq m$. 
Set $B=\bigcap_{i=1}^m \pi(B_i)$ and $B_i'=B_i \cap \pi^{-1}(B)$.
We have $B \cap \pi(X)=\bigcup_{i=1}^m \pi(B'_i \cap X)$, and it is definable.
Therefore, $\pi(X)$ is locally definable.
\end{proof}

We need the following curve selection lemma:
\begin{lemma}[Curve selection lemma]\label{lem:curve}
Consider a locally o-minimal expansion of the group of reals $(\mathbb R, <,0,+)$. 
Let $X$ be a locally definable subset of $\mathbb R^n$ and $a \in \partial X$.
There exist $\varepsilon>0$ in $\mathbb R$ and a definable continuous curve $\gamma:]0,\varepsilon[ \rightarrow X$ such that $\displaystyle\lim_{t \to 0} \gamma(t)=a$.
\end{lemma}
\begin{proof}
Take a sufficiently small open box $B$ containing the point $a$.
The set $B \cap X$ is definable.
We may assume that $X$ is definable considering $B \cap X$ in place of $X$.
The proof is the same as the o-minimal case given in \cite[p.94]{vdD} using Corollary \ref{cor:dim2} in place of the monotonicity theorem for o-minimal structures.
We omit the proof.
\end{proof}

\begin{lemma}\label{lem:o-min}
Consider a locally o-minimal structure whose universe is the set of reals $\mathbb R$.
Then, it is o-minimal or there exists an unbounded discrete definable subset of $\mathbb R$.
\end{lemma}
\begin{proof}
We have only to show that the structure is o-minimal if every discrete definable subset of $\mathbb R$ is bounded.

Let $A$ be a definable subset of $\mathbb R$ and $A'$ be the boundary of the set $A$.
The definable set $A'$ is discrete.
It is bounded by the assumption.
Take a bounded closed interval  $I$ containing $A'$.
It is obvious that $A'$ is a finite set because $I$ is compact and the structure is locally o-minimal.
The definable set $A'$ is finite and $A$ is a finite union of points and open intervals.
It means that the structure is o-minimal.
\end{proof}

We next define multi-cells.
\begin{definition}
Consider a locally o-minimal expansion of the group of reals $(\mathbb R, <,0,+)$. 
We define a \textit{multi-cell} $X$ in $\mathbb R^n$ inductively.
\begin{itemize}
\item If $n=1$, $X$ is a discrete definable set or all connected components of the definable set $X$ are open intervals. 
\item When $n>1$, let $\pi:\mathbb R^n \rightarrow \mathbb R^{n-1}$ be the projection forgetting the last coordinate.
The projection image $\pi(X)$ is a multi-cell and, for any connected component $Y$ of $X$, $\pi(Y)$ is a connected component of $\pi(X)$ and $Y$ is one of the following forms:
\begin{align*}
Y&=\pi(Y) \times \mathbb R   \text{,}\\
Y&=\{(x,y) \in \pi(Y) \times \mathbb R \;|\; y=f(x)\} \text{,}\\
Y &= \{(x,y) \in \pi(Y) \times \mathbb R \;|\; y>f(x)\} \text{,}\\
Y &= \{(x,y) \in \pi(Y) \times \mathbb R \;|\; y<g(x)\} \text{ and }\\
Y &= \{(x,y) \in \pi(Y) \times \mathbb R \;|\; f(x)<y<g(x)\}
\end{align*}
for some continuous functions $f$ and $g$ defined on $\pi(Y)$ with $f<g$.
\end{itemize} 
\end{definition}

The proof of the main theorem in this section is long.
We divide the proof into several lemmas.
\begin{lemma}\label{lem:multi-cell-pre}
Consider a locally o-minimal expansion of the group of reals $(\mathbb R, <,0,+)$ which is not o-minimal.
Let $X$ be a definable subset of $\mathbb R^n$ and $\pi:\mathbb R^n \rightarrow \mathbb R^{n-1}$ be the projection forgetting the last coordinate.
Assume that, for any $x \in \mathbb R^{n-1}$, the fiber $X \cap \pi^{-1}(x)$ is at most of dimension zero.
Then, there exist definable subsets $Z_1$ and $Z_2$ of $\mathbb R^{n-1}$ satisfying the following conditions:
 \begin{enumerate}
\item[(a)]  $\dim(Z_1) < \dim(\pi(X))$ and $\dim(Z_2) < \dim(\pi(X))$;
\item[(b)] For any $x \in \mathbb R^n \setminus \pi^{-1}(Z_1)$, there exists an open box $U$ containing the point $x$ such that $X \cap U=\emptyset$ or $\pi(X) \cap \pi(U)$ is a manifold and $X \cap U$ is the graph of a continuous function defined on $\pi(X) \cap \pi(U)$;
\item[(c)] Any connected component $C$ of $X \setminus \pi^{-1}(Z_1 \cup Z_2)$ is bounded in the last coordinate;
\item[(d)] $\partial C \subset \pi^{-1}(Z_1 \cup Z_2)$ for any connected component $C$ of $X \setminus \pi^{-1}(Z_1 \cup Z_2)$.
\end{enumerate}
\end{lemma}
\begin{proof}
We first find a definable subset $Z_1$ of $\mathbb R^{n-1}$ with $\dim(Z_1) < \dim(\pi(X))$ satisfying the condition (b).
Set $d=\dim(X)$.
We have $\dim(\pi(X))=d$ by Corollary \ref{cor:dim1}.
The notation $\operatorname{Reg}(\pi(X))$ denotes the set of points at which $\pi(X)$ is locally a $d$-dimensional manifold.
It is open in $\pi(X)$.
The notation $\operatorname{Sing}(\pi(X))$ denotes the singular locus given by $\pi(X) \setminus \operatorname{Reg}(\pi(X))$.
It is a definable set of dimension smaller than $d$ by \cite[Theorem 4.2]{Fuji}.
Let $S$ be the set of points $x$ in $\mathbb R^n \setminus \pi^{-1}(\operatorname{Sing}(\pi(X)))$ at which there exist no open boxes $U$ containing the point $x$ such that $X \cap U=\emptyset$ or $X \cap U$ is the graph of a continuous function defined on $\pi(X) \cap \pi(U)$.
We have $\dim(S) < d$ by \cite[Theorem 4.2]{Fuji}.
Set $Z_1 = \overline{\pi(S)} \cup \operatorname{Sing}(\pi(X))$.
We have $\dim Z_1 \leq \max\{\dim(S), \operatorname{Sing}(\pi(X))\}< \dim(X)$ by \cite[Theorem 5.5]{Fuji} and Theorem \ref{thm:dim}.
The condition (b) is obviously satisfied.
We may assume that $Z_1=\emptyset$ by considering $X \setminus \pi^{-1}(Z_1)$ in place of $X$.

There exists an unbounded discrete definable subset $D$ of $\mathbb R$ by Lemma \ref{lem:o-min}.
We may assume that $\inf(D)=-\infty$ and $\sup(D)=\infty$ by considering $D \cup (-D)$ in place of $D$ because the group operation is definable.
Let $Z_r$ be the boundary of $X \cap (\pi(X) \times \{r\})$ in $\pi(X) \times \{r\}$ for any $r \in D$.
Set $W=\partial X \cup \left( \displaystyle\bigcup_{r \in D} Z_r \right)$ and $Z_2=\pi(W)$.
The set $W$ is definable.
In fact, $\displaystyle\bigcup_{r \in D} Z_r$ is given by $\{(x,r) \in \mathbb R^{n-1} \times \mathbb R\;|\; r \in D, \ \text{the point }x \text{ is contained in the boundary of } X \cap \pi^{-1}(r) \text{ in }\pi^{-1}(r)\}$ and it is definable.
We prove the assertion (a) for $Z_2$.
Let $C$ be a definable cell contained in $\displaystyle\bigcup_{r \in D} Z_r$. 
There exists $r \in D$ with $C \subset Z_r$.
We have $\dim (C) < \dim (X \cap (\pi(X) \times \{r\})) \leq \dim(X)$ and $\dim(\partial X)<\dim(X)$ by \cite[Theorem 5.5]{Fuji}.
We also have $\dim\left(\displaystyle\bigcup_{r \in D} Z_r\right)<\dim(X)$ by \cite[Corollary 5.3]{Fuji}.
We finally get the assertion (a) by Corollary \ref{cor:dim1} and \cite[Corollary 5.4(ii)]{Fuji}.

Furthermore, the following assertion $(*)$ holds true because $X \cap \left((\pi(X) \setminus Z_2) \times D\right)$ is locally the graph of a constant function on $\pi(X) \cap \pi(U)$.
\medskip

$(*)$: There exists an open neighborhood of $X \cap \left((\pi(X) \setminus Z_2) \times D\right)$ in $(\pi(X) \setminus Z_2) \times D$ which does not intersect with $X \setminus \pi^{-1}(Z_2)$ other than $X \cap \left((\pi(X) \setminus Z_2) \times D\right)$.
\medskip

We next show that the definable set $Z_2$ satisfies the assertions (c) and (d).
We consider two cases, separately.
We first consider the case in which $C \cap ((\pi(X) \setminus Z_2) \times D)$ has a non-empty interior in $(\pi(X) \setminus Z_2) \times D$.
We have $C \subset (\pi(X) \setminus Z_2) \times \{r\}$ for some $r \in D$.
In fact, the definable set $X \cap ((\pi(X) \setminus Z_2) \times \{r\}) $ is open and closed in $(\pi(X) \setminus Z_2) \times \{r\}$ by the definition of $Z_2$ for any $r \in D$.
We have $C \subset \mathbb R^{n-1} \times \{r\}$ for some $r \in D$ because $C$ is connected.
In particular, $C$ is bounded on the last coordinate and satisfies the assertion (c).
The assertion (d) is also trivial in this case because $C$ is closed and open subset of $X \cap ((\pi(X) \setminus Z_2) \times \{r\}) $ in this case.

The next case is the case in which $C \cap ((\pi(X) \setminus Z_2) \times D)$ has an empty interior in $\mathbb R^{n-1} \times D$.
We have $C \cap ((\pi(X) \setminus Z_2) \times D)=\emptyset$ by the assertion $(*)$ in this case.
We demonstrate that $C \subset \mathbb R^{n-1} \times ]r_1,r_2[$ for some $r_1,r_2 \in \mathbb R$.
Let $\pi_2:\mathbb R^n \rightarrow \mathbb R$ be the projection onto the last coordinate.
Set $s = \inf(\pi_2(C))$.
Assume that $s=-\infty$.
Take a point $x_1 \in C$, then there exists an $R \in D$ with $\pi_2(x_1) > R$ because $\inf(D)=-\infty$.
We can get $x_2 \in C$ with $\pi_2(x_2)<R$ because $s=-\infty$.
The partition $C = \{x \in C\;|\; \pi_2(x)>R\} \cup \{x \in C\;|\; \pi_2(x)<R\}$ is a partition into two non-empty open and closed subsets.
It contradicts the assumption that $C$ is connected.
Therefore, we have $\inf(\pi_2(C)) > -\infty$.
We can show that $\sup(\pi_2(C)) < + \infty$ in the same way.
We can take $r_1, r_2 \in \mathbb R$  with $r_1 < \inf(\pi_2(C)) \leq \sup(\pi_2(C)) < r_2$.
We have finished the proof of the assertion (c).

We finally demonstrate the assertion (d).
Assume the contrary.
Take an arbitrary point $x \in \partial C \setminus \pi^{-1}(Z_2)$.
We have $x \not\in \mathbb R^{n-1} \times D$ by the assertion $(*)$.
Set $X'=X\setminus \pi^{-1}(Z_2)$.
The point $x$ is contained in $X'$ because $X'$ is closed in $\mathbb R^n \setminus \pi^{-1}(Z_2)$ by the definition of $W$ and $Z_2$.
There exists an open box $B$ containing the point $x$ and a definable cell decomposition of $B$ partitioning $X \cap B$, $D \cap B$, $\pi^{-1}(Z_2)$ and $W$ by \cite[Theorem 4.2]{Fuji}.
Since $C$ is a connected component of $X'$, $C \cap B$ is a finite union of cells.
Let $C_1$ be a cell contained in $C$ with $x \in \overline{C_1}$.
Since $C$ is closed in $X'$, we have $x \in \overline{C_1} \cap X' \subset C$.
Contradiction to the assumption that $x \in \partial C$.
\end{proof}

The following lemma is the major induction step of the proof of the main theorem.

\begin{lemma}\label{lem:multi-cell-pre2}
Consider a locally o-minimal expansion of the group of reals $(\mathbb R, <,0,+)$ which is not o-minimal.
Let $X$ be a definable subset of $\mathbb R^n$ and $\pi:\mathbb R^n \rightarrow \mathbb R^{n-1}$ be the projection forgetting the last coordinate.
Assume that, for any $x \in \mathbb R^{n-1}$, the fiber $X \cap \pi^{-1}(x)$ is at most of dimension zero.
Assume further that any definable subset of $\mathbb R^{n-1}$ is partitioned into finite multi-cells.
Then, the definable set $X$ is also partitioned into finite multi-cells.
Furthermore, the projection images of two distinct multi-cells are disjoint.
\end{lemma}
\begin{proof}
We prove the lemma by the induction on $\dim(X)$.
When $\dim(X)=0$, $X$ is a discrete definable set and its projection images are also discrete by Theorem \ref{thm:dim}.
Therefore, X itself is a multi-cell.
We consider the case in which $\dim(X)>0$.
We can find definable subsets $Z_1$ and $Z_2$ of $\mathbb R^{n-1}$ satisfying the following conditions by Lemma \ref{lem:multi-cell-pre}.
 \begin{enumerate}
\item[(a)]  $\dim(Z_1) < \dim(\pi(X))$ and $\dim(Z_2) < \dim(\pi(X))$;
\item[(b)] For any $x \in \mathbb R^n \setminus \pi^{-1}(Z_1)$, there exists an open box $U$ containing the point $x$ such that $X \cap U=\emptyset$ or $\pi(X) \cap \pi(U)$ is a manifold and $X \cap U$ is the graph of a continuous function defined on $\pi(X) \cap \pi(U)$;
\item[(c)] Any connected component $C$ of $X \setminus \pi^{-1}(Z_1 \cup Z_2)$ is bounded in the last coordinate;
\item[(d)] $\partial C \subset \pi^{-1}(Z_1 \cup Z_2)$ for any connected component $C$ of $X \setminus \pi^{-1}(Z_1 \cup Z_2)$.
\end{enumerate}

The lemma holds true for $X \cap \pi^{-1}(Z_1 \cup Z_2)$ by the induction hypothesis because $\dim(X \cap \pi^{-1}(Z_1 \cup Z_2)) =\dim(Z_1 \cup Z_2) < \dim(\pi(X))=\dim(X)$ by Corollary \ref{cor:dim1}.
Replacing $X$ with $X \setminus \pi^{-1}(Z_1 \cup Z_2)$, we may further assume that any connected component of $X$ is bounded in the last coordinate and closed in $\pi^{-1}(\pi(X))$.
We can partition $\pi(X)$ into finite multi-cells by the assumption.
Hence, we may assume that $\pi(X)$ is a multi-cell.
We demonstrate that $X$ is a multi-cell in this case.
Let $C$ be a connected component of $X$.
We have only to show the following assertions:
\begin{itemize}
\item $\pi(C)$ is a connected component of $\pi(X)$.
\item $C$ is the graph of a continuous function defined on $\pi(C)$.
\end{itemize}

We first demonstrate that $\pi(C)$ is a connected component of $\pi(X)$.
The connected component $C$ is locally definable by Lemma \ref{lem:ld1}.
The image $\pi(C)$ is locally definable by Lemma \ref{lem:ld2} because $C$ is bounded in the last coordinate.
There exists a connected component $E$ of $\pi(X)$ with $\pi(C) \subset E$ because $\pi(C)$ is connected.
Assume that $\pi(C) \not=E$.
Take a point $x$ in the boundary of $\pi(C)$ in $E$.
There exists a definable continuous curve $\gamma:]0,\varepsilon[ \rightarrow \pi(C)$ with $\displaystyle\lim_{t \to 0}\gamma(t) = x$ by Lemma \ref{lem:curve}.
Define $f_u:]0,\varepsilon[ \rightarrow \mathbb R$ by $f_u(t)=\sup\{y \in \mathbb R\;|\; (\gamma(t),y) \in C\}$.
The definable set $\{(t,y) \in ]0,\varepsilon[ \times \mathbb R\;|\; (\gamma(t),y) \in C\}$ is definable  by Lemma \ref{lem:ld1}(a) because $C$ is bounded in the last coordinate.
Therefore, the function $f_u$ is definable.
We may assume that $f_u$ is continuous by Corollary \ref{cor:dim2} by taking a sufficiently small $\varepsilon>0$ if necessary.
The limit $y=\displaystyle\lim_{t \to 0} f_u(t)$ exists.
In fact, the frontier of the graph is of dimension zero by \cite[Theorem 5.5]{Fuji}.
The intersection of the frontier with the line $t=0$ is a singleton because the function $f_u$ is continuous.
It means that the limit exists.
We have $(x,y) \in C$ because $C$ is closed in $X \cap \pi^{-1}(E)$.
By the assumption, there exists an open box $U$ with $(x,y) \in U$ such that $U \cap C$ is the graph of a continuous function defined on $E \cap \pi(U)$.
Therefore, the image $\pi(C)$ contains the neighborhood $E \cap \pi(U)$ of the point $x$.
Contradiction to the assumption that $x$ is a point in the boundary of $\pi(C)$ in $E$.

We next demonstrate that $C$ is the graph of a continuous function defined on $\pi(C)$. 
We have only to show that the restriction of $\pi$ to $C$ is injective because $X$ is locally the graph of a continuous function by the assumption.
Set 
\begin{equation*}
T=\{x \in \pi(C)\;|\; |\pi^{-1}(x) \cap C| > 1\}\text{.}
\end{equation*}
We have only to demonstrate that $T$ is an empty set.
We first show that $T$ is locally definable. 
Consider the set $S=\{(x,y_1,y_2) \in \mathbb R^{n-1} \times \mathbb R \times \mathbb R\;|\; (x,y_1) \in C \text{, } (x,y_2) \in C \text{ and } y_1 < y_2\}$. 
The locally definable set $S$ is bounded in the last coordinate, and the image $S'$ of $S$ under the projection forgetting the last coordinate is locally definable by Lemma \ref{lem:ld2}.
It is obvious that $S'$ is also bounded in the last coordinate and $T=\pi(S')$.
The set $T$ is locally definable using Lemma \ref{lem:ld2} again.

The set $T$ is open in $\pi(C)$.
In fact, take an arbitrary point $x \in T$.
There exist $y_1<y_2 \in \mathbb R$ with $(x,y_1),(x,y_2) \in C$.
Since $X$ is locally the graph of a continuous function, there exists an open box $B$ with $x \in B \cap \pi(C)$ such that $X \cap \pi^{-1}(B)$ contains the graphs of continuous functions whose values at $x$ are $y_1$ and $y_2$, respectively.
Therefore, $B \cap \pi(C)$ is contained in $T$, and $T$ is open in $\pi(C)$.

We next show that $T$ is closed in $\pi(C)$.
Assume the contrary.
Take a point $x \in \pi(C) \cap \partial T$.
We can take the unique $y \in \mathbb R$ with $(x,y) \in C$ because $x \not\in T$.
There exists a definable continuous curve $\gamma:]0,\varepsilon[ \rightarrow \pi(C) \cap T$ such that $\displaystyle\lim_{t \to 0}\gamma(t)=x$ by Lemma \ref{lem:curve}.
We define the maps $\eta_u,\eta_l:]0,\varepsilon[ \rightarrow \mathbb R$ by
\begin{align*}
&\eta_u(t)=\sup\{u \in \mathbb R\;|\; (\gamma(t),u) \in C\} \text{ and }\\
&\eta_l(t)=\inf\{u \in \mathbb R\;|\; (\gamma(t),u) \in C\} \text{.}
\end{align*}
They are well-defined because $C$ is bounded in the last coordinate.
Take a sufficiently small $\varepsilon>0$.
They are definable and continuous and they have the limits $y_u = \displaystyle\lim_{t \to 0} \eta_u(t) \in \mathbb R$ and $y_l = \displaystyle\lim_{t \to 0} \eta_l(t) \in \mathbb R$ for the same reason as above.
We have $\eta_u(t) \not=\eta_l(t)$ because $\gamma(t) \in T$.
We have $(x,y_u) \in C$ and $(x, y_l) \in C$ because $C$ is closed in $\pi^{-1}(\pi(C))$.
We therefore get $y=y_l=y_u$ because $x \not\in T$.
The definable set $X$ is not locally the graph of a definable function at $(x,y)$ because $\eta_u(t) \not=\eta_l(t)$.
Contradiction.
We have shown that $T$ is closed in $\pi(C)$.

Since $\pi(C)$ is connected and $T$ is open and closed in $\pi(C)$, we have $T=\pi(C)$ or $T=\emptyset$.
We have only to lead to a contradiction assuming that $T=\pi(C)$.
Define the function $f_u:\pi(C) \rightarrow \mathbb R$ by $f_u(x)=\sup\{t\;|\; (x,t) \in C\}$.
We can easily show that its graph is a locally definable set using Lemma \ref{lem:ld1}(a) because $C$ is bounded in the last coordinate.
It is a continuous function.
In fact, let $\mathcal D$ be the set of all the points at which $f_u$ is discontinuous.
Take a point $x \in \mathcal D$.
The set $V_x$ is the intersection of $\pi^{-1}(x)$ with the closure of the graph of $f_u|_{\pi(C) \setminus \{x\}}$, where $f_u|_{\pi(C) \setminus \{x\}}$ denote the restriction of $f_u$ to $\pi(C) \setminus \{x\}$.
The closure of the graph of $f_u|_{\pi(C) \setminus \{x\}}$ is locally definable by Lemma \ref{lem:ld1}(b).
The set $V_x$ is locally definable and compact.
Consequently, $V_x$ is definable by Lemma \ref{lem:ld1}(a).
There exists a point $(x,y) \in V_x$ with  $y \not= f_u(x)$ by the assumption.
Note that $(x,y) \in C$ because $C$ is closed in $\pi^{-1}(\pi(C))$.
The set $C$ is locally the graphs of continuous functions $g$ and $h$ defined on a neighborhood of $x$ in $\pi(X)$ at $(x,y)$ and $(x,f_u(x))$, respectively. 
Take a sufficiently small $\varepsilon >0$. 
Since $g$ and $h$ are continuous and $g(x)<h(x)$, we have $g(x')+\varepsilon<h(x')$ if $x'$ is sufficiently close to $x$.
We also get $h(x') \leq f_u(x')$ by the definition of the function $f_u$. 
We then have $g(x')+\varepsilon < f_u(x')$ for any $x'$ sufficiently close to $x$ and 
we obtain $(x,y)=(x,g(x)) \not\in V_x$.
Contradiction.
We have demonstrated that the function $f_u$ is continuous.
Consider the graph $\{(x,y) \in C\;|\; y = f_u(x)\}$.
It is easy to prove that the graph is an open and closed proper subset of $C$ using the fact that $C$ is locally the graph of a continuous function.
Contradiction to the assumption that $C$ is connected.
\end{proof}

The following theorem is the main theorem in this section.
\begin{theorem}\label{thm:multi-cell}
Consider a locally o-minimal expansion of the group of reals $(\mathbb R, <,0,+)$.
A definable set is partitioned into finite multi-cells.
\end{theorem}
\begin{proof}
Consider the case in which the structure in consideration is o-minimal.
A definable set is partitioned into finite cells by \cite[Theorem 3.2.11]{vdD}.
It is also a partition into finite multi-cells because a cell is simultaneously a multi-cell.

We next consider the case in which the structure is not o-minimal.
Let $X$ be a definable subset of $\mathbb R^n$.
We demonstrate that the set $X$ is partitioned into finite multi-cells.
We prove it by the induction on $n$.
Consider the case in which $n=1$.
The theorem is clear when $X=\emptyset$ or $X=\mathbb R$.
We consider the other cases.
The set $X_1$ is the union of all the maximal open intervals contained in $X$, which is definable.
In fact, the set $X_1$ is described as follows:
\begin{align*}
X_1 &=\{x \in X\;|\; \exists \varepsilon >0,\ \forall y \in \mathbb R,\ |x-y| < \varepsilon \rightarrow y \in X\} \text{.}
\end{align*}
The set $X_2 = X \setminus X_1$ is the set of the isolated points and the endpoints of the maximal open intervals in $X$ because the structure is locally o-minimal.
It is clearly a discrete definable set.
The decomposition $X=X_1 \cup X_2$ is a partition into multi-cells. 

We next consider the case in which $n>1$.
Let $\pi:\mathbb R^n \rightarrow \mathbb R^{n-1}$ be the projection forgetting the last coordinate.
Consider the sets 
\begin{align*}
&X_{\text{oi}}=\{(x,y) \in \mathbb R^{n-1} \times \mathbb R \;|\; \exists \varepsilon >0,\  \forall y',\  |y'-y|<\varepsilon \rightarrow (x,y') \in X\} \text{,}\\
&X_{\forall}=\{(x,y) \in \mathbb R^{n-1} \times \mathbb R \;|\; \forall y',\  (x,y') \in X\} \text{,}\\
&X'_{\infty}=\{(x,y) \in \mathbb R^{n-1} \times \mathbb R \;|\; \exists \varepsilon >0, \ \forall y',\ y'>y-\varepsilon \rightarrow (x,y') \in X\} \text{ and }\\
&X'_{-\infty}=\{(x,y) \in \mathbb R^{n-1} \times \mathbb R \;|\;\exists \varepsilon >0, \ \forall y',\ y'<y+\varepsilon \rightarrow (x,y') \in X\} \text{.}
\end{align*}
Set 
\begin{align*}
&X_{\text{boi}}=X_{\text{oi}} \setminus (X_{\forall} \cup X'_{\infty} \cup X'_{-\infty})\text{,}\\ &X_{\infty} = X'_{\infty} \setminus X_{\forall}\text{,}\\
&X_{-\infty} = X'_{\infty} \setminus X_{\forall}\text{ and }\\
&X_{\text{pt}} = X \setminus (X_{\text{boi}} \cup X_{\infty} \cup  X_{-\infty} \cup X_{\forall})\text{.}
\end{align*}
The definable set $X$ is partitioned as follows:
\begin{equation*}
X = X_{\text{boi}} \cup X_{\infty} \cup  X_{-\infty} \cup X_{\forall} \cup X_{\text{pt}}\text{.}
\end{equation*}
By the definition, connected components of non-empty fibers of $X_{\text{boi}}$,  $X_{\infty}$,  $X_{-\infty}$, $X_{\forall}$ and $X_{\text{pt}}$ are a bounded open interval, an open interval unbounded above and bounded below, an open interval bounded above and unbounded below, $\mathbb R$ and a point, respectively.

We have only to show that the above five definable sets are partitioned into multi-cells.
The definable set $X_{\text{pt}}$ is partitioned into multi-cells by Lemma \ref{lem:multi-cell-pre2}.
As to $X_{\forall}$, there exists a partition into multi-cells $\pi(X_{\forall})=\bigcup_{i=1}^k Y_i$ by the induction hypothesis.
Set $X_{\forall,i}= Y_i \times \mathbb R$, then the partition $X_{\forall}=\bigcup_{i=1}^k X_{\forall,i}$ is a partition into multi-cells.
Consider the set 
\begin{equation*}
Y_{\infty}=\{(x,y) \in \pi(X_{\infty}) \times \mathbb R\;|\; (x,y) \not\in X_{\infty},\ \forall y',\ y'>y \rightarrow (x,y') \in X_{\infty}\}\text{.}
\end{equation*}
The definable sets $Y_{\infty}$ consists of the lower endpoints of fibers of $X_{\infty}$.
In particular, $Y_\infty$ satisfies the assumption of Lemma \ref{lem:multi-cell-pre2}.
Let $Y_\infty=\bigcup_{i=1}^k Y_{\infty,i}$ be a partition into multi-cells given by Lemma \ref{lem:multi-cell-pre2}.
Set $X_{\infty,i}=X_{\infty} \cap \pi^{-1}(\pi(Y_{\infty,i}))$.
The definable set $X_{\infty,i}$ is a multi-cell.
In fact, it is clear that the projection image $\pi(X_{\infty,i})$ is a multi-cell because $\pi(X_{\infty,i})=\pi(Y_{\infty,i})$.
Since $Y_{\infty,i}$ is a multi-cell, it is the graph of a continuous function $f$ defined on $\pi(Y_{\infty,i})$.
It is obvious that $X_{\infty,i}=\{(x,y) \in \pi(X_{\infty,i}) \times \mathbb R\;|\; y>f(x)\}$ by the definition.
Hence, the definable set $X_{\infty,i}$ is a multi-cell, and the partition $X_{\infty}=\bigcup_{i=1}^k X_{\infty,i}$ is a partition into multi-cells.
We can show that the definable set $X_{-\infty}$ is partitioned into multi-cell in the same way.

The remaining task is to demonstrate that $X_{\text{boi}}$ is partitioned into multi-cells.
We may assume the followings:
\begin{enumerate}[(i)]
\item All the connected components of non-empty fibers of $X$ are bounded open intervals; 
\item For any $x \in \pi(X)$, the closures of two distinct connected components of $X \cap \pi^{-1}(x)$ have an empty intersection.
\end{enumerate}
We may employ the assumption (i) by setting $X=X_{\text{boi}}$.
We demonstrate that we may also employ the assumption (ii). 
Consider the definable set 
\begin{align*}
Y_{\text{both}} &= \{(x,y_1,y_2) \in \pi(X) \times \mathbb R^2\;|\; (x,y_1) \not\in X, (x,y_2) \not\in X, y_1<y_2,\\
&\qquad \forall c,\ y_1 < c < y_2 \rightarrow (x,c) \in X\}\text{.}
\end{align*}
Set 
\begin{align*}
X_{\text{upper}} &= \{(x,y) \in \pi(X) \times \mathbb R\;|\; \exists y_1, y_2, \ (x,y_1,y_2) \in Y_{\text{both}},\  (y_1+y_2)/2 <y <y_2\}\text{,}\\
X_{\text{middle}} &= \{(x,y) \in \pi(X) \times \mathbb R\;|\; \exists y_1, y_2, \ (x,y_1,y_2) \in Y_{\text{both}}, \  y=(y_1+y_2)/2\}\text{ and }\\
X_{\text{lower}} &= \{(x,y) \in \pi(X) \times \mathbb R\;|\; \exists y_1, y_2, \ (x,y_1,y_2) \in Y_{\text{both}},\  y_1<y<(y_1+y_2)/2\}\text{.}
\end{align*}
The definable set $X_{\text{middle}}$ can be partitioned into finite multi-cells by Lemma \ref{lem:multi-cell-pre2}.
The closures of two distinct connected components of $X_{\text{upper}} \cap \pi^{-1}(x)$ have empty intersections for all $x \in \pi(X)$.
The fiber $X_{\text{lower}} \cap \pi^{-1}(x)$ also enjoys the same property.
Therefore, we may assume that the definable set $X$ satisfies the assumption (ii) by setting $X=X_{\text{upper}}$ and $X=X_{\text{lower}}$.
\medskip

Consider the definable sets 
\begin{align*}
Y_{\text{upper}} &= \{(x,y) \in \pi(X) \times \mathbb R\;|\; (x,y) \not\in X,\  \exists \varepsilon >0,\ \forall c,\ y-\varepsilon < c < y\\
&\qquad \rightarrow (x,c) \in X\}\text{ and }\\
Y_{\text{lower}} &= \{(x,y) \in \pi(X) \times \mathbb R\;|\; (x,y) \not\in X,\  \exists \varepsilon >0,\ \forall c,\ y < c < y+\varepsilon\\
&\qquad \rightarrow (x,c) \in X\}\text{.}
\end{align*}
For any $x \in \pi(X)$, the fiber $Y_{\text{upper}} \cap \pi^{-1}(x)$ is the set of the upper endpoints of the maximal open intervals contained in $X \cap \pi^{-1}(x)$ by the assumption (i).
The fiber $Y_{\text{lower}} \cap \pi^{-1}(x)$ is the set of the lower endpoints of the maximal open intervals.
By Lemma \ref{lem:multi-cell-pre2}, both $Y_{\text{upper}}$ and  $Y_{\text{lower}}$ are partitioned into finite multi-cells.
Let $Y_{\text{upper}}=\bigcup_{i=1}^k Y_{\text{upper},i}$ and $Y_{\text{lower}}=\bigcup_{i=1}^l Y_{\text{lower},i}$ be partitions into finite multi-cells.
We have $\pi(Y_{\text{upper},i_1}) \cap \pi(Y_{\text{upper},i_2})=\emptyset$ by Lemma \ref{lem:multi-cell-pre2} if $i_1 \not=i_2$.
We may further assume that, for all $1 \leq i \leq k$ and $1 \leq j \leq l$, we have $\pi(Y_{\text{upper},i})=\pi(Y_{\text{lower},j})$ or  $\pi(Y_{\text{upper},i}) \cap \pi(Y_{\text{lower},j}) = \emptyset$.
In fact, for all $1 \leq i \leq k$ and $1 \leq j \leq l$, the definable set $\pi(Y_{\text{upper},i}) \cap \pi(Y_{\text{lower},j})$ are partitioned into finite multi-cells by the induction hypothesis.
Let $\pi(Y_{\text{upper},i}) \cap \pi(Y_{\text{lower},j})=\bigcup_{m=1}^{p(i,j)} Z_{ijm}$ be partitions.
Set $Y_{\text{upper},ijm}=Y_{\text{upper},i} \cap \pi^{-1}(Z_{ijm})$ and $Y_{\text{lower},ijm}=Y_{\text{lower},j} \cap \pi^{-1}(Z_{ijm})$.
They are obviously multi-cells satisfying the requirement.

Set $X_i = X \cap \pi^{-1}(\pi(Y_{\text{upper},i}))$.
We have a partition $X=\bigcup_{i=1}^k X_i$.
The remaining task is to show that $X_i$ is a multi-cell.
Take an arbitrary connected component $C$ of $X_i$ and an arbitrary point $\hat{z} \in C$.
Set $\hat{x}=\pi(\hat{z})$ and $\hat{z}=(\hat{x},\hat{y})$ for some $\hat{y} \in \mathbb R$.
Since connected components of the fiber $X \cap \pi^{-1}(\hat{x})$ are bounded open intervals by the assumption (i), there exist $y_u, y_l \in \mathbb R$, $1 \leq i' \leq k$ and $1 \leq j' \leq l$ with $y_l < \hat{y} < y_u$, $(x,y_u) \in Y_{\text{upper},i'}$, $(x,y_l) \in Y_{\text{lower},j'}$ and $(\hat{x},y) \in X$ for all $y_l<y<y_u$.
We have $\pi(Y_{\text{upper},i'})=\pi(Y_{\text{lower},j'})$ by the assumption.
Let $Z$ be its connected component containing the point $\hat{x}$.
There are two continuous function $f$ and $g$ defined on $Z$ such that $y_l=f(\hat{x})$, $y_u=g(\hat{x})$ and the graphs of $f$ and $g$ are connected components of $Y_{\text{lower},j'}$ and $Y_{\text{upper},i'}$, respectively, because $Y_{\text{lower},j'}$ and $Y_{\text{upper},i'}$ are multi-cells.
We demonstrate that $f(x)<g(x)$ on $Z$ and $C=\{(x,y) \in Z \times \mathbb R\;|\; f(x)<y<g(x)\}$.

We show that the graph of $f$ does not intersect with $Y_{\text{upper}}$. 
In particular, we have $f(x)<g(x)$ on $Z$ by the intermediate value theorem.
Assume the contrary.
Let $x' \in Z$ and $y'=f(x')$ with $(x',y') \in Y_{\text{upper}}$. 
By the definition of $f$ and $Y_{\text{upper}}$, there exist $y_1,y_2 \in \mathbb R$ with $y_1<y'<y_2$ such that $\{x\} \times ]y_1,y'[$ and $\{x\} \times ]y',y_2[$ are connected components of the fiber $X \cap \pi^{-1}(x)$.
The intersection of their closures is not empty.
It is a contradiction to the assumption (ii).

We finally show that $C=\{(x,y) \in Z \times \mathbb R\;|\; f(x)<y<g(x)\}$.
The set $C$ is contained in $\{(x,y) \in Z \times \mathbb R\;|\; f(x)<y<g(x)\}$ because the intersection of the latter set with $X$ is closed and open in $X$ by the definition.
We demonstrate the opposite inclusion.
Assume the contrary.
Let $(x',y')$ be a point satisfying $x' \in Z$, $f(x')<y'<g(x')$ and $(x',y') \not\in C$.
By the assumption (i), there exists $\overline{y} \in \mathbb R$ with $f(x')<\overline{y}<y'$ and 
$(x',\overline{y}) \in Y_{\text{upper}}$.
Since we have $\pi(Y_{\text{upper},i_1}) \cap  \pi(Y_{\text{upper},i_2})=\emptyset$ for all $i_1 \not= i_2$, we have $(x',\overline{y}) \in Y_{\text{upper},i'}$.
Since $Y_{\text{upper},i'}$ is a multi-cell, the connected component of $Y_{\text{upper},i'}$ containing the point $(x',\overline{y})$ is the graph of some continuous function $g'$ defined on $Z$.
We have $f(x')<g'(x')<g(x')$.
The graph of $g'$ does not intersect with the graph of $g$ because $Y_{\text{upper},i'}$ is a multi-cell.
The graph of $g'$ does not intersect with the graph of $f$ because the graph of $f$ does not intersect with $Y_{\text{upper}}$ as we demonstrated previously. 
We get $y_l=f(\hat{x})<g'(\hat{x})<g(\hat{x})=y_u$ by the intermediate value theorem.
We obtain $(\hat{x},g'(\hat{x})) \not\in X$, which contradicts the fact that $(\hat{x},y) \in X$ for all $y_l<y<y_u$.
\end{proof}

\section{Uniform local definable cell decomposition}\label{sec:udcd}

In this section, we first show that a locally o-minimal expansion of the group of reals $(\mathbb R, <,0,+)$ has a uniformity property.
We also prove the uniform local definable cell decomposition theorem introduced in Section \ref{sec:intro} using the uniformity property.

We need the following technical definition for proving the uniformity theorem.

\begin{definition}
Consider a locally o-minimal expansion of the group of reals $(\mathbb R, <,0,+)$.
Let $X \subset \mathbb R^n$ be a multi-cell and $Y$ be a discrete definable subset of $X$.
The notation $\pi_k:\mathbb R^n \rightarrow \mathbb R^k$ denotes the projection onto the first $k$ coordinates for all $1 \leq k \leq n$.
Note that  $\pi_n$ is the identity map.
The definable set $Y$ is a \textit{representative set of connected components of $X$} if the intersection of $\pi_k(Y)$ with any connected component of $\pi_k(X)$ is a singleton for any $1 \leq k \leq n$. 
\end{definition}

\begin{lemma}\label{lem:onept}
Consider a locally o-minimal expansion of the group of reals $(\mathbb R, <,0,+)$.
Let $X \subset \mathbb R^{m+n}$ be a multi-cell and $\pi:\mathbb R^{m+n} \rightarrow \mathbb R^m$ be the projection onto the first $m$ coordinates.
There exists a definable subset $Y$ of $X$ such that $Y \cap \pi^{-1}(x)$ is a representative set of  connected components of $X \cap \pi^{-1}(x)$ for any $x \in \pi(X)$. 
\end{lemma}
\begin{proof}
We demonstrate the lemma by the induction on $n$.
We first consider the case in which $n=1$.
Consider the following definable sets:
\begin{align*}
&S_{\infty} = \{x \in \pi(X)\;|\; \forall y \in \mathbb R, (x,y) \in X\}\text{,}\\
&S_u = \{x \in \pi(X)\;|\; \exists y \in \mathbb R, \ \forall z, \ z > y \rightarrow (x,z) \in X\} \text{ and }\\
&S_l = \{x \in \pi(X)\;|\; \exists y \in \mathbb R, \ \forall z, \ z < y \rightarrow (x,z) \in X\}\text{.}
\end{align*}
The definable functions $\rho_u:S_u\setminus S_\infty \rightarrow \mathbb R$ and $\rho_l:S_l\setminus S_\infty \rightarrow \mathbb R$ are given as follows:
\begin{align*}
\rho_u(x) &= \inf\{y \in \mathbb R\;|\; \forall z, \ z > y \rightarrow (x,z) \in X\}\text{ and }\\
\rho_l(x) &= \sup\{y \in \mathbb R\;|\; \forall z, \ z < y \rightarrow (x,z) \in X\}\text{.}
\end{align*}
We set
\begin{align*}
Y_c &= \{(x,y_1,y_2) \in \pi(X) \times \mathbb R^2\;|\; (x,y_1) \not\in X, (x,y_2) \not\in X, y_1<y_2,\\
&\qquad \forall c,\ y_1 < c < y_2 \rightarrow (x,c) \in X\}\text{ and }\\
Y_p &= \{(x,y) \in X\;|\; \exists \varepsilon > 0,\  \forall c,\ 0<|y-c|<\varepsilon \rightarrow (x,c) \not\in X\}\text{.}
\end{align*}
We finally set
\begin{align*}
Y &= \{(x,\rho_u(x)+\varepsilon) \in \mathbb R^{m+1}\;|\; x \in S_u\} \cup \{(x,\rho_l(x)-\varepsilon) \in \mathbb R^{m+1}\;|\; x \in S_l\}\\
&\quad  \cup \{(x,y) \in \mathbb R^{m+1}\;|\; \exists y_1, y_2, \ (x,y_1,y_2) \in Y_c,\ y = (y_1+y_2)/2\}\\
&\quad  \cup Y_p \cup (S_\infty \times \{0\})\text{,}
\end{align*}
where $\varepsilon$ is a fixed positive real number.
The definable set $Y \cap \pi^{-1}(x)$ is obviously a representative set of connected components of $X \cap \pi^{-1}(x)$ for any $x \in \pi(X)$. 

We consider the case in which $n>1$.
The notations $\pi_1:\mathbb R^{m+n} \rightarrow \mathbb R^{m+n-1}$ and $\pi_2: \mathbb R^{m+n-1} \rightarrow \mathbb R^{m}$ denote the projections forgetting the last coordinate and onto the first $m$ coordinates, respectively.
The projection image $\pi_1(X)$ is a multi-cell by the definition of multi-cells.
There exists a definable subset $Y_1 \subset \pi_1(X)$ such that the definable set $Y_1 \cap \pi_2^{-1}(x)$ is a representative set of connected components of $\pi_1(X) \cap \pi_2^{-1}(x)$ for any $x \in \pi(X)$ by applying the induction hypothesis to $\pi_1(X)$ and $\pi_2$.
Set $X'=X \cap \pi_1^{-1}(Y_1)$, and apply the lemma for $n=1$ to $X'$ and $\pi_1$.
We can find a representative set $Y$ of connected components of $X'$.
It is easy to demonstrate that $Y$ is also a representative set of connected components of $X$.
\end{proof}

\begin{theorem}[Uniformity theorem]\label{thm:uniform}
Consider a locally o-minimal expansion of the group of reals $(\mathbb R, <,0,+)$.
For any definable subset $X$ of $\mathbb R^{n+1}$, there exist a positive element $r \in \mathbb R$ and a positive integer $K$ such that, for any $a \in \mathbb R^n$, the definable set $X \cap (\{a\} \times ]-r,r[)$ has at most $K$ connected components.
\end{theorem}
\begin{proof}
%
%
Consider the set
\begin{equation*}
\mathcal X = \{(r,x,y) \in \mathbb R \times \mathbb R^n \times \mathbb R\;|\; (x,y) \in X\text{, }r>0\text{ and }-r<y<r\}\text{.}
\end{equation*}
Apply Theorem \ref{thm:multi-cell} to $\mathcal X$.
We have a partition into multi-cells $\mathcal X = \bigcup_{i=1}^k \mathcal X_i$.
Let $\Pi_1:\mathbb R \times \mathbb R^n \times \mathbb R \rightarrow \mathbb R$ be the projection onto the first coordinate.
We next apply Lemma \ref{lem:onept} to $\mathcal X_i$ and $\Pi_1$.
For any $1 \leq i \leq k$, we can take a definable subset $\mathcal Y_i$ of $\mathcal X_i$ such that, for any $r \in \mathbb R$, $\mathcal Y_i \cap \Pi_1^{-1}(r)$ is a representative set of connected components of $\mathcal X_i \cap \Pi_1^{-1}(r)$.
Let $\Pi_2:\mathbb R \times \mathbb R^n \times \mathbb R \rightarrow \mathbb R^2$ be the projection given by $\Pi_2(r,x,y)=(r,y)$.
Set $\mathcal Z_i = \Pi_2(\mathcal Y_i)$.
We have $\dim(\mathcal Y_i \cap \Pi_1^{-1}(r)) \leq 0$ for all $r \in \mathbb R$ because $\mathcal Y_i \cap \Pi_1^{-1}(r)$ is discrete.
We get $\dim(\mathcal Z_i \cap (\{r\} \times \mathbb R)) \leq 0$ by Theorem \ref{thm:dim}.
Since the structure is a uniformly locally o-minimal structure of the second kind, there exists a positive element $R \in \mathbb R$ such that, for any $r >0$ with $r<R$, the definable sets $\mathcal Z_i \cap (\{r\} \times ]-R,R[)$ consist of finite points for all $1 \leq i \leq k$.

Fix an $r >0$ with $r<R$ and an arbitrary point $a \in \mathbb R^n$.
The definable set $\mathcal Z_i \cap (\{r\} \times  ]-r,r[ )$ is a finite set for any $1 \leq i \leq k$.
Set $K=\displaystyle\sum_{i=1}^k \left|\mathcal Z_i \cap (\{r\} \times  ]-r,r[ )\right|$.
Let $\Pi_3:\mathbb R^{n+2} \rightarrow \mathbb R^{n+1}$ be the projection forgetting the first coordinate.
Set 
\begin{align*}
&X^{<r} =\{(x,y) \in X\;|\; -r < y < r\}=\Pi_3(\mathcal X \cap (\{r\} \times \mathbb R^{n+1}))
\text{,}\\
&X_i = \Pi_3(\mathcal X_i \cap (\{r\} \times \mathbb R^{n+1}))\text{ and }\\
&Y_i = \Pi_3(\mathcal Y_i \cap (\{r\} \times \mathbb R^{n+1}))\text{.}
\end{align*}
Let $\pi_1:\mathbb R^{n+1} \rightarrow \mathbb R^n$ and $\pi_2:\mathbb R^{n+1} \rightarrow \mathbb R$ be the projections onto first $n$ coordinates and onto the last coordinate, respectively.
We easily get $X^{<r} = \bigcup_{i=1}^k X_i$ and $X \cap (\{a\} \times ]-r,r[) = X^{<r} \cap \pi_1^{-1}(a)$.
The definable set $X_i$ is a multi-cell.
The definable set $Y_i$ is a representative set of connected components of $X_i$.
The notation $p:\mathbb R^2 \rightarrow \mathbb R$ denotes the projection onto the second coordinate.
We have $\pi_2(Y_i) = p( \mathcal Z_i \cap (\{r\} \times  ]-r,r[))$.

When $a \in \pi_1(X_i)$ for some $1 \leq i \leq k$, there exists a point $a' \in \pi_1(Y_i)$ contained in the connected component of $\pi_1(X_i)$ containing the point $a$.
The definable set $X_i \cap \pi_1^{-1}(a)$ has the same number of connected components as $X_i \cap \pi_1^{-1}(a')$, and which is equal to $|Y_i \cap \pi_1^{-1}(a')|$ 
by the definitions of multi-cells and representative sets of their connected components.
The notation $\operatorname{NC}(S)$ denotes the number of connected components of a definable set $S$.
We therefore have
\begin{align*}
\operatorname{NC}(X \cap (\{a\} \times ]-r,r[)) &= \operatorname{NC}(X^{<r} \cap \pi_1^{-1}(a))
\leq \displaystyle\sum_{i=1}^k \operatorname{NC}(X_i \cap \pi_1^{-1}(a))\\
&=\displaystyle\sum_{i=1}^k \operatorname{NC}(X_i \cap \pi_1^{-1}(a'))
=\displaystyle\sum_{i=1}^k |Y_i \cap \pi_1^{-1}(a')|\\
& \leq \displaystyle\sum_{i=1}^k |\pi_2(Y_i)| = \displaystyle\sum_{i=1}^k \left|p( \mathcal Z_i \cap (\{r\} \times  ]-r,r[)) \right|\\
&=\displaystyle\sum_{i=1}^k \left|\mathcal Z_i \cap (\{r\} \times  ]-r,r[)\right| =K\text{.}
\end{align*}
We have finished the proof.
\end{proof}

We begin with the proof of Theorem \ref{thm:main}.
\begin{definition}
Let $(j_1, \ldots, j_d)$ be an increasing sequence of positive integers with $1 \leq j_k \leq n$ for all $1 \leq k \leq d$.
Let $\pi_{(j_1, \ldots, j_d)}: \mathbb R^n \rightarrow \mathbb R^{d}$ be the projection given by $\pi_{(j_1, \ldots, j_d)}(x_1,\ldots,x_n)=(x_{j_1}, \ldots, x_{j_d})$.
A cell $C$ in $\mathbb R^n$ of dimension $d$ is \textit{of type} $(j_1, \ldots, j_d)$ if the restriction of the projection  $\pi_{(j_1, \ldots, j_d)}$ to $C$ is a definable hemeomorphism onto its image.
\end{definition}

\begin{theorem}[Uniform local definable cell decomposition]\label{thm:udcd}
Consider a locally o-minimal expansion of the group of reals $(\mathbb R, <,0,+)$. 
Let $\{A_\lambda\}_{\lambda\in\Lambda}$ be a finite family of definable subsets of $\mathbb R^{m+n}$.
There exist an open box $B$ in $\mathbb R^n$ containing the origin and a finite partition of definable sets 
\begin{equation*}
\mathbb R^m \times B = X_1 \cup \ldots \cup X_k
\end{equation*}
such that $B=(X_1)_b \cup \ldots \cup (X_k)_b$ is a definable cell decomposition of $B$ for any $b \in \mathbb R^m$ and $X_i \cap A_\lambda = \emptyset$ or $X_i \subset A_\lambda$ for any $1 \leq i \leq k$ and $\lambda \in \Lambda$.
Furthermore, the type of the cell $(X_i)_b$ is independent of the choice of $b$ with $(X_i)_b \not= \emptyset$.
\end{theorem}
\begin{proof}
We first show the assertion for $n=1$.
For any definable set $S \subset \mathbb R^{m+1}$,  the notation $\operatorname{bd}_m(S)$ denotes the set $\{(x,y) \in \mathbb R^m \times \mathbb R\;|\; y \in \operatorname{bd}(S_x)\}$.
Since the locally o-minimal expansion of the group of reals is strongly locally o-minimal by \cite[Corollary 3.4]{TV}, there exists an open interval $I$ containing the origin such that $S \cap I$ is a finite union of points and open interval for any definable subset $S$ of $\mathbb R$.
Set $X= \bigcup_{\lambda \in \Lambda}\operatorname{bd}_m(A_\lambda \cap I)$.
The fibers $X_b$ are finite sets for all $b \in \mathbb R^m$.
It is obvious that any definable cell decomposition of $B$ partitioning $X_b \cap I$ partitions $\{(A_\lambda)_b \cap I\}_{\lambda\in\Lambda}$ for any point $b \in \mathbb R^m$.

There exists a positive integer $K$ and a positive element $r \in \mathbb R$ such that $|X_b \cap(\{b\} \times  ]-r,r[)| \leq K$ for any point $b \in \mathbb R^m$ by Theorem \ref{thm:uniform}.
Set $I= ]-r,r[$ and $S_i=\{b \in \mathbb R^m\;|\; |X_b \cap I|=i\}$ for all $0 \leq i \leq K$.
The family $\{S_i\}_{i=0}^K$ partitions the parameter space $\mathbb R^m$.
Let $y_j(b)$ be the $j$-th largest point of $X_b \cap I$ for all $b \in S_i$.
Set $y_0(b)=-r$ and $y_{i+1}(b)=r$ for all $b \in S_i$.
Applying Corollary \ref{cor:dim2} inductively, we can find a partition into definable sets 
\begin{equation*}
S_i = S_{i0} \cup \ldots \cup S_{im}
\end{equation*}
such that $S_{ik}=\emptyset$ or $\dim(S_{ik})=k$, and $y_j$ is continuous on $S_{ik}$ for any $0 \leq j \leq i$ and $0 \leq k \leq m$.
We set
\begin{align*}
&C_{ijk} =\{(x,y_j(x)) \in S_{ik} \times \mathbb R\} \ \ (1 \leq j \leq  i)\text{ and }\\
&D_{ijk} = \{(x,y) \in S_{ik} \times \mathbb R\;|\; y_j(x) < y < y_{j+1}(x)\} \ \ (0 \leq j \leq i)
\end{align*}
for any $0 \leq i \leq K$ and $0 \leq k \leq m$. 
Consider the family of maps $\mathcal F = \{\sigma: \Lambda \rightarrow \{0,1\}\}$.
Set 
\begin{align*}
&T^0_{ijk\sigma} =\{x \in S_{ik}\;|\; C_{ijk} \cap (\{x\} \times \mathbb R) \text{ is contained in } A_\lambda \text{ iff } \sigma(\lambda)=1\} \text{ and }\\
&T^1_{ijk\sigma} =\{x \in S_{ik}\;|\; D_{ijk} \cap (\{x\} \times \mathbb R) \text{ is contained in } A_\lambda \text{ iff } \sigma(\lambda)=1\}
\end{align*}
for any $0 \leq i \leq K$, $0 \leq j \leq i$, $0 \leq k \leq m$ and $\sigma \in \mathcal F$.
We finally set $C_{ijk\sigma}=C_{ijk} \cap (T^0_{ijk\sigma} \times \mathbb R)$ and $D_{ijk\sigma}=D_{ijk} \cap (T^1_{ijk\sigma} \times \mathbb R)$.
The partition 
\begin{equation*}
\mathbb R^m \times I = \bigcup_{i=0}^K \left(\bigcup_{k=1}^m \left(\bigcup_{\sigma \in \mathcal F}\left(\bigcup_{j=1}^i C_{ijk\sigma} \cup \bigcup_{j=0}^i D_{ijk\sigma}\right)\right)\right)
\end{equation*}
 is the desired partition.
Furthermore, the above definable functions $y_j$ can be chosen as continuous functions on $p(C_{ijk\sigma})$ and $p(D_{ijk\sigma})$, where $p:\mathbb R^{m+1} \rightarrow \mathbb R^m$ is the projection forgetting the last coordinate.
It is clear that the type of the cell $(X_i)_b$ is independent of the choice of $b$ with $(X_i)_b \not= \emptyset$.

We consider the case in which $n>1$.
Let $\pi:\mathbb R^{m+n} \rightarrow \mathbb R^{m+n-1}$ be the projection forgetting the last coordinate.
Applying the theorem for $n=1$ to the family $\{A_\lambda\}_{\lambda\in\Lambda}$, there exist an open interval $I$ containing the origin and a partition $\mathbb R^{m+n-1} \times I = Y_1 \cup \ldots \cup Y_l$ such that  $I=(Y_1)_b \cup \ldots \cup (Y_l)_b$ is a definable cell decomposition $I$ for any $b \in \mathbb R^{m+n-1}$ and $Y_i \subset A_\lambda$ or $Y_i \cap A_\lambda=\emptyset$ for any $1 \leq i \leq l$ and $\lambda \in \Lambda$. 
We can further assume that $Y_i$ is one of the following forms:
\begin{align*}
Y_i &= \{(x,f(x)) \in \pi(Y_i) \times \mathbb R\} \text{ and }\\
Y_i &= \{(x,y) \in \pi(Y_i) \times \mathbb R\;|\; f(x)<y<g(x)\}\text{,}
\end{align*}
where $f$ and $g$ are definable continuous functions on $\pi(Y_i)$ with $f<g$.

Apply the induction hypothesis to the family $\{\pi(Y_i)\}_{i=1}^l$.
There exist an open box $B'$ in $\mathbb R^{n-1}$ containing the origin and a partition $\mathbb R^{m} \times B' = Z_1 \cup \ldots \cup Z_p$ such that  $B'=(Z_1)_b \cup \ldots \cup (Z_p)_b$ is a definable cell decomposition $B'$ for any $b \in \mathbb R^{m}$, and $\pi(Y_i) \cap Z_j = \emptyset$ or $Z_j \subset \pi(Y_i)$ and the type of the cell $(Z_j)_b$ is independent of the choice of $b$ with $(Z_j)_b \not= \emptyset$.

Set $B=B' \times I$ and $X_{ij}=Y_i \cap \pi^{-1}(Z_j)$ for all $1 \leq i \leq l$ and $1 \leq j \leq p$.
Let $\{X_i\}_{i=1}^k$ be the family of non-empty $X_{ij}$'s.
It is easy to demonstrate that the family $\{X_i\}_{i=1}^k$ satisfies the requirement of the theorem.
We omit the proof. 
\end{proof}

\begin{corollary}
Consider a locally o-minimal expansion of the group of reals $(\mathbb R, <,0,+)$. 
For any definable subset $X$ of $\mathbb R^n$, there exist a positive integer $K$ and an open box $B$ in $\mathbb R^n$ containing the origin such that the definable set $X \cap (b+B)$ has at most $K$ connected components for all $b \in \mathbb R^n$.
Here, $b+B$ denotes the set given by $\{x \in \mathbb R^n\;|\; x-b \in B\}$.
\end{corollary}
\begin{proof}
Consider the definable set $Y$ defined by 
\begin{equation*}
\{(y,x) \in \mathbb R^n \times \mathbb R^n\;|\; x-y \in X\} \text{.} 
\end{equation*}
Applying Theorem \ref{thm:udcd}, there exist an open box $B$ containing the origin and a partition $\mathbb R^{n} \times B = X_1 \cup \ldots \cup X_K$ such that  $B=(X_1)_b \cup \ldots \cup (X_K)_b$ is a definable cell decomposition $B$ partitioning the definable set $Y_b \cap B$ for any $b \in \mathbb R^{n}$.
It means that the definable set $X \cap (b+B)$ is the union of at most $K$ cells.
The set $X \cap (b+B)$ has at most $K$ connected components because cells are connected.
We have finished the proof.
\end{proof}

We begin to prove Theorem \ref{thm:main}.
\begin{lemma}\label{lem:last}
Let $\mathcal M=(M,<,\ldots)$ be a densely linearly ordered structure.
Consider a definable set $C \subset M^n$ defined by a first-order formula with parameter $\overline{c}$.
There exists a first-order sentence with parameters $\overline{c}$ expressing the condition for $C$ being a definable cell of type $(j_1, \ldots, j_d)$.
\end{lemma}
\begin{proof}
We prove the lemma by the induction on $n$.
When $n=1$, the definable set $C$ is a cell if and only if $C$ is a point or an open interval.
This condition is clearly expressed by a first-order sentence.

We next consider the case in which $n>1$.
The notation $\pi:M^n \rightarrow M^{n-1}$ denotes the projection forgetting the last factor.
The condition for $\pi(C)$ being a cell is represented by a first order sentence with parameters $\overline{c}$ by the induction hypothesis.
We only prove the lemma in the case in which the definable set $C$ is of the form
\begin{equation*}
C=\{(x,y) \in M^{n-1} \times M\;|\; f(x) < y < g(x)\}\text{,}
\end{equation*}
where $f$ and $g$ are definable continuous functions defined on $\pi(C)$.
We can demonstrate the lemma in the other cases in the similar way.
The above condition is equivalent to the following conditions:
\begin{itemize}
\item For any $x \in \pi(C)$, the fiber $C_x=\{y \in M\;|\; (x,y) \in C\}$ is a bounded interval.
\item Set $f(x)=\inf\{y \in M\;|\;(x,y) \in C\}$ and $g(x)=\sup\{y \in M\;|\;(x,y) \in C\}$ for any $x \in \pi(C)$, then $f$ and $g$ are continuous on $\pi(C)$.
\end{itemize}
The above conditions are obviously expressed by first-order sentences with parameters $\overline{c}$.
\end{proof}

The following corollary is Theorem \ref{thm:main}.
\begin{corollary}\label{cor:last}
Theorem \ref{thm:udcd} holds true for a structure elementarily equivalent to a locally o-minimal expansion of the group of reals. 
\end{corollary}
\begin{proof}
Consider a structure $\mathcal M=(M,+,0,<,\ldots)$ elementarily equivalent to a locally o-minimal expansion $\widetilde{\mathbb R}$ of the group of reals. 
We first reduce to the case in which $A_\lambda$ are definable without parameters for all $\lambda \in \Lambda$.
There exist parameters $\overline{c} \in M^p$ and first-order formulae $\varphi_{\lambda}(x,y,\overline{c})$ with parameters $\overline{c}$ defining the definable sets $A_{\lambda}$ for all $\lambda \in \Lambda$.
Set $A'_\lambda =\{(z,x,y) \in M^p \times M^m \times M^n\;|\; \mathcal M \models \varphi_\lambda(x,y,z)\}$.
If the corollary holds true for the family $\{A'_\lambda\}_{\lambda \in \Lambda}$, the corollary also holds true for the family $\{A_\lambda\}_{\lambda \in \Lambda}$ because $A_\lambda$ is the fiber $(A'_\lambda)_{\overline{c}}=\{(x,y) \in M^m \times M^n\;|\; (\overline{c},x,y) \in A'_\lambda\}$.
Hence, we may assume that $A_\lambda$ are definable without parameters for all $\lambda \in \Lambda$.
The notations $\varphi_\lambda(x,y)$ denote the first-order formulae without parameters defining the definable sets $A_\lambda$.

Let $A_\lambda^{\mathbb R}$ be the definable subset of $\mathbb R^{m+n}$ defined by the formula $\varphi_\lambda(x,y)$ for each $\lambda \in \Lambda$.
By Theorem \ref{thm:udcd}, there exist an open box $B^{\mathbb R}$ in $\mathbb R^n$ containing the origin and a partition into definable sets 
\begin{equation*}
\mathbb R^m \times B^{\mathbb R} = X_1^{\mathbb R} \cup \ldots \cup X_k^{\mathbb R}
\end{equation*}
such that the fibers $(X_i^{\mathbb R})_b$ are definable cells of a fixed type for all $b \in \mathbb R^m$ with $(X_i^{\mathbb R})_b \not= \emptyset$ and $X_i^{\mathbb R} \subset A_\lambda^{\mathbb R}$ or $X_i^{\mathbb R} \cap A_\lambda^{\mathbb R} = \emptyset$ for all $1 \leq i \leq k$.
There exist parameters $\overline{c} \in \mathbb R^p$ and first-order formulae $\psi_i(x,y,\overline{c})$ with parameters $\overline{c}$ defining the definable sets $X_i^{\mathbb R}$ for all $1 \leq i \leq k$.

Using the first-order formulae $\psi_i(x,y,\overline{c})$, the condition that
\begin{itemize}
\item there exists an open box $B^{\mathbb R}$ in $\mathbb R^n$ containing the origin and
\item $\mathbb R^m \times B^{\mathbb R} = X_1^{\mathbb R} \cup \ldots \cup X_k^{\mathbb R}$
\end{itemize}
can be expressed by a first-order sentence $\Phi(\overline{c})$ with parameters $\overline{c}$.
Let $\Psi_i(\overline{c})$ be the sentence expressing the condition $X_i^{\mathbb R} \subset A_\lambda^{\mathbb R}$ or $X_i^{\mathbb R} \cap A_\lambda^{\mathbb R} = \emptyset$ for any $1 \leq i \leq k$.
The condition for the fiber $(X_i^{\mathbb R})_b$ being a cell for all $b \in \mathbb R^m$ is expressed by a first-order formula $\Pi_i(\overline{c})$ with parameters $\overline{c}$ by Lemma \ref{lem:last}.
We have
\begin{equation*}
\widetilde{\mathbb R} \models \Phi(\overline{c}) \wedge \bigwedge_{i=1}^k \left( \Psi_i(\overline{c}) \wedge \Pi_i(\overline{c}) \right)
\end{equation*}
by the definitions of $\Phi(\overline{c})$, $\Psi_i(\overline{c})$ and $\Pi_i(\overline{c})$.
We therefore get
\begin{equation*}
\widetilde{\mathbb R} \models \exists \overline{c}\ \Phi(\overline{c}) \wedge \bigwedge_{i=1}^k \left( \Psi_i(\overline{c}) \wedge \Pi_i(\overline{c}) \right) \text{.}
\end{equation*}
Since $\mathcal M$ is elementarily equivalent to $\widetilde{\mathbb R}$, we finally obtain 
\begin{equation*}
\mathcal M \models \exists \overline{d}\ \Phi(\overline{d}) \wedge \bigwedge_{i=1}^k \left( \Psi_i(\overline{d}) \wedge \Pi_i(\overline{d}) \right) \text{.}
\end{equation*}
Take $d \in M^p$ satisfying the above condition and set $X_i=\{(x,y) \in M^m \times M^n\;|\; \mathcal M \models \psi_i(x,y,\overline{d})\}$ for all $1 \leq i \leq k$.
Then, there exists an open box $B$ in $M^n$ containing the origin such that the partition $M^m \times B=X_1 \cup \ldots \cup X_k$ is the desired partition.
\end{proof}

\end{document}